\documentclass[12pt,letterpaper,oneside]{amsart}
          \makeatletter
          \def\@setcopyright{}
          \def\serieslogo@{}
          \makeatother

\usepackage{xcolor}
\definecolor{mg}{rgb}  {0.85, 0.,  0.85}

\usepackage[backend=biber,bibencoding=utf8,citestyle=alphabetic,bibstyle=alphabetic,sorting=nyt,isbn=false,doi=false,url=false,hyperref=true,backref=true,backrefstyle=three,giveninits=true]{biblatex}
\renewbibmacro*{volume+number+eid}{%
	\textbf{\printfield{volume}}%
	\setunit*{\adddot}
	\setunit*{\addnbspace}
	\printfield{number}%
	\setunit{\addcomma\space}%
	\printfield{eid}}
\renewbibmacro{in:}{%
	\ifentrytype{article}{}{\printtext{\bibstring{in}\intitlepunct}}}
\DeclareFieldFormat[article]{number}{\mkbibparens{#1}}
\usepackage[UKenglish]{babel}
\usepackage[colorlinks]{hyperref}
\usepackage{amsmath,amssymb,amsthm}
\usepackage{mathrsfs}
\usepackage{verbatim}
\usepackage[utf8]{inputenc}
\usepackage[T1]{fontenc}
\usepackage{lmodern}
\usepackage{graphicx}
\usepackage{caption}
\usepackage{subcaption}
\usepackage{wrapfig}
\usepackage{microtype}
\usepackage{bm}
\usepackage{dsfont}
\usepackage[normalem]{ulem} 
\usepackage[retainorgcmds]{IEEEtrantools}
\usepackage[showframe=false]{geometry}
\usepackage[capitalise,nameinlink]{cleveref}
\usepackage{tikz-cd}
\usetikzlibrary{arrows}
\usetikzlibrary{positioning}

\usepackage{csquotes}
\usepackage{mathtools}
\usepackage{soul}
\usepackage{xfrac}
\allowdisplaybreaks
\crefformat{equation}{(#2#1#3)}
\crefrangeformat{equation}{(#3#1#4) to~(#5#2#6)}
\crefmultiformat{equation}{(#2#1#3)}%
{ and~(#2#1#3)}{, (#2#1#3)}{ and~(#2#1#3)}

\newcommand\ie{{\em i.e.}~}

\def\N{\mathbb{N}}
\def\T{\mathbb{T}}

\def\Z{\mathbb{Z}}
\def\C{\mathbb{C}}

\def\U{\mathscr U}
\def\F{\mathscr F}

\def\A{\mathcal A}

\def\F{\mathcal F}

\def\R{\mathbb{R}}

\def\V{\mathcal V}

\def\({\left(}
\def\[{\left[}
\def\){\right)}
\def\]{\right]}
\def\l|{\left\lvert}
\def\r|{\right\rvert}

\DeclareMathOperator*{\slim}{s\hspace{0.1pt}-\hspace{0.1pt}lim}

\def\<{\langle}
\def\>{\rangle}

\def\e{\mathrm e}

\def\Ve{\hat V^\pm_\kappa}

\DeclareMathOperator{\rank}{Rank}
\DeclareMathOperator{\Ker}{Ker}

\newtheorem{Theorem}{Theorem}[section]
\newtheorem{Remark}[Theorem]{Remark}
\newtheorem{Lemma}[Theorem]{Lemma}

\newtheorem{Proposition}[Theorem]{Proposition}
\newtheorem{Definition}[Theorem]{Definition}

\crefname{Lemma}{Lemma}{Lemmas}

\newcommand{\bel}{\begin{equation} \label}
\newcommand{\ee}{\end{equation}}
\newcommand{\ba}{\begin{array}}
\newcommand{\ea}{\end{array}}

\begin{document}

\title[Eigenvalue Asymptotics near a flat band]{Eigenvalue Asymptotics near a flat band in presence of a slowly decaying potential}

\author[P.\ Miranda]{Pablo Miranda}\address{Departamento de Matem\'atica y Ciencia de la Computaci\'on, Universidad de Santiago de Chile, Las Sophoras 173, Santiago, Chile.}\email{pablo.miranda.r@usach.cl}
\author[D.\ Parra]{Daniel Parra}
\address{Departamento de Matemática y Estadística, Universidad de La Frontera, Avenida Francisco Salazar 01145, Temuco, Chile}
\email{daniel.parra@ufrontera.cl}

\begin{abstract}
We provide eigenvalue asymptotics for a Dirac--type operator on $\Z^n$, $n\geq 2$, perturbed by multiplication operators that decay as $|\mu|^{-\gamma}$ with $\gamma<n$. We show that the eigenvalues accumulate near the value of the flat band at a ``semiclassical'' rate with a constant that encodes the 
structure of the flat band. Similarly, we show that this behaviour can be obtained also for a Laplace operator on a periodic graph. 
\bigskip


\bigskip

\end{abstract}

\maketitle \vspace{-1cm}


\section{Introduction}\label{sec:intro}

In this article, we consider an operator whose decomposition into a direct integral presents a flat band. We  are interested in the accumulation of eigenvalues near the value of the flat band when a perturbation is added. We start by briefly explaining the setting and then discuss our motivation for entertaining such an analysis.

Let us denote by $\mathcal{X}$ the standard graph structure in $\Z^n$ and consider the Dirac type operator on $\ell^2(\mathcal{X})$ defined by 

\begin{equation*}
H_0=\begin{pmatrix}
m & d^*\\
d & -m
\end{pmatrix},
\end{equation*}
where $d$ is the discrete version of the exterior derivative and $m$ a positive constant. We refer to \Cref{subsec:DDo} for the precise definition but one can readily notice that by construction $H_0$ satisfies the supersymmetry condition making it into an \emph{abstract Dirac operator} as in \cite{Th92}.  Moreover, from the analysis of its band functions, see \eqref{eq:bandas} below, we obtain that the spectrum of $H_0$ is 
\begin{equation}\label{spectrum}
\sigma(H_0)=\sigma_{ess}(H_0)=\sigma_{ac}(H_0) 
=[-\sqrt{m^2+4n},-m]\bigcup[m,\sqrt{m^2+4n}] \ .
\end{equation}  
An essential observation pertinent to this study is that if $n\geq2$, $-m$  is an embedded infinite dimensional eigenvalue of $H_0$. \emph{Throughout this article, we will assume that $m>0$ and $n\geq2$.}

Let us now consider a perturbation by a multiplication operator $V:\mathcal{X}\to \R$ decaying at infinity. 
Hence we define 
\begin{equation}\label{eq:fullH}
H:=H_0+ V\ .
\end{equation}
Since $V$ is a compact operator, $\sigma_{\rm ess}(H)=\sigma_{\rm ess}(H_0)$. Moreover, since  $m>0$, equality  \cref{spectrum} tell  us that  $(-m,m)$  is a gap in the essential spectrum of $H$. Then, for $\lambda\in(0,m)$  we consider the function 
\begin{equation*}
\mathcal{N}(\lambda)={\rm Rank} \mathds{1}_{(-m+\lambda,0)}(H)\ ,
\end{equation*}
with $\mathds{1}_{\Omega}$ being the characteristic function over the Borel set $\Omega$. Clearly, this function count the number of eigenvalues of $H$ (with multiplicity)  on the interval $(-m+\lambda,0)$. 

Our primary objective is to analyze the asymptotic behaviour of $\mathcal N$  as $\lambda\downarrow 0$ for a specific class of perturbations that decay slowly at infinity. Further details can be found in \Cref{Def:admissible}, while our main result is presented in \Cref{T1}.

One motivation for studying $\mathcal{N}$ stems from our previous work on the distribution of eigenvalues as presented in \cite{MPR23}. This article extends our prior research in three significant ways: it encompasses the general $n$-dimensional scenario, incorporates the potential for non-definite perturbations, and addresses potentials with slower rates of decay at infinity. Moreover, we employ a distinct method to derive the effective Hamiltonian, drawing inspiration from the analysis of eigenvalue distributions for magnetic Schrödinger operators, see \cite{Ra90,IwTa98,PuRo11}. This approach yields an effective Hamiltonian with a ``typical'' structure denoted as $PVP$, where $P$ is a projection. 

Another motivation arises from the recent surge in interest surrounding the study of flat bands in the discrete setting. Unlike the common assumption in the continuous case, periodic Schrödinger operators in periodic graphs often exhibit flat bands, as discussed in \cite{SY23}. While these configurations have long been studied by the physics community, see \cite{BL13,KFSH20} and references therein, recent attention from the spectral theory community has also emerged, see for instance \cite{KTW23,PS23x,Zw24,GZ23x}. Remarkably, we demonstrate a striking similarity between the results obtained for our Dirac operator and those for the Laplacian on a specific periodic graph showcasing such a flat band, see \Cref{theo:nonmaintheo}.

We finish this introduction by briefly describing the structure of the article. In \Cref{sec:model} we give the precise definition of $H_0$ and study its main spectral characteristics. In \Cref{sec:main} we introduce the class of admissible perturbations and state our main result, which we prove in \Cref{sec:proof}. Finally, in \Cref{sec:Laplace}, we show a similar result for the standard graph--Laplacian in a particular $\Z^2$--periodic graph.

\section{Spectral Theory for a Dirac operator on \texorpdfstring{$\Z^n$}{Zn} }\label{sec:model}

In this section we provide the definition of $H_0$ taking most notations from \cite{Pa17}, see also \cite{CT15}, recall its integral decomposition, and show the explicit expression of its resolvent as a fibered operator that will be central to our investigations. 

\subsection{Discrete Dirac operator}\label{subsec:DDo}

We denote by $\mathcal{X}=(\mathcal V,\mathcal A)$ the standard graph structure in $\Z^n$. That is, the set of vertices $\mathcal V$ consists of points  $\mu\in\Z^n$ and the set of oriented edges $\mathcal A$ is composed of pairs $(\mu,\nu)$ such that $\nu=\mu\pm \delta_j$, where $\{\delta_j\}_{j=1}^n$ denotes the canonical basis of $\Z^n$. An edge in $\mathcal{A}$ is written $\e=(\mu,\nu)$ and its transpose  $\overline{\e}:=(\nu,\mu)$. Let us consider the vector spaces of {$0-$cochains} $C^0(\mathcal{X})$ and {$1-$cochains} $C^1(\mathcal{X})$ given by
\begin{equation*}
C^0(\mathcal{X}):=\{f:\mathcal{V}\to\C\}\text{ ; }\qquad C^1(\mathcal{X}):=\{g: \mathcal{A}\to\C\mid g(\e)=-g(\overline{\e})\}.
\end{equation*}
The Hilbert spaces $\ell^2(\mathcal V)$ and $\ell^2(\mathcal A)$  are naturally defined by the inner products of cochains: 
$
\langle f_1,f_2 \rangle_0 =\sum_{\mu\in \mathcal V}f_1(\mu)\overline{f_2(\mu)}$ and $\langle g_1,g_2 \rangle_1=\frac12 \sum_{e\in \mathcal A}g_1(e)\overline{g_2(e)}$, respectively.

The {coboundary operator} $d:\ell^2(\mathcal V)\to \ell^2(\mathcal A)$ is defined by
\begin{equation}\label{1}
df(\e):=f(\nu)-f(\mu),\quad\text{ for }\e=(\mu,\nu)\in\mathcal A\ .
\end{equation}
This is the discrete version of the exterior derivative and its adjoint $d^*: \ell^2(\mathcal A)\to \ell^2(\mathcal V)$ is given at each edge by the finite sum 
\begin{equation}\label{2bis}
d^*g(\mu)=\sum_{j=1}^ng(\mu,\mu\pm\delta_j) ,\quad\text{ for } \mu\in\mathcal V.
\end{equation}
Let us define the Hilbert space $\ell^2(\mathcal{X})=\ell^2(\mathcal V)\oplus\ell^2(\mathcal A)$ and denote by $P_{\mathcal V}$ and $P_{\mathcal A}$ the corresponding projections. Further, we introduce the involution $\tau$ on $\ell^2(\mathcal{X})$ by
\begin{equation*}
\tau(f,g)=(f,-g)\ .
\end{equation*}
Then, for a strictly positive constant $m$ let us consider the free Dirac operator 
\begin{equation*}
H_0=d+d^*+m\tau=
\begin{pmatrix}
0 & d^*\\
d & 0
\end{pmatrix}+m\begin{pmatrix}
1 & 0\\
0 & -1
\end{pmatrix}=\begin{pmatrix}
m & d^*\\
d & -m
\end{pmatrix}
\end{equation*}
where we have slightly abused notation by considering $d$ and $d^*$ acting on $\ell^2(\mathcal{X})$. Note that $H_0$ is a Dirac-type operator in the sense that 
\begin{equation*}
H_0^2=\begin{pmatrix}
\Delta_0+m^2 & 0\\
0 & \Delta_1+m^2
\end{pmatrix}
\end{equation*}
where $\Delta_0$ is the  Laplacian on vertices and $\Delta_1$ is  the (1-down) Laplacian on edges.

\subsection{Integral decomposition}\label{subsec:Decom}

Let us denote by $\mathcal H:=L^2(\T^n,\C^{n+1})$. In this section we construct a unitary operator $\U:l^2(\mathcal{X})\to \mathcal H$. Consider the action of $\Z^n$ on $\mathcal X$ given for $\mu\in\Z^n$, $x\in \V$, and $\e=(x,y)\in A$ by   \begin{equation*}
\mu x=\mu+x \text{ and } \mu\e:=(\mu+x,\mu+y)\ .
\end{equation*} 
Then, a natural class of representatives of the orbits of such action is given by $\mathbf{0}\in \V$ together with the edges $\e_j=(\mathbf{0},\delta_j)$ and $\e_j^-=(\mathbf{0},-\delta_j)$. 

 Let us denote $\T^n=\R^n/[0,1]^n$ and set $C_c(\mathcal{X})$ to be the set of cochains with compact support, \ie, $f\in C_c(\mathcal{X})$ if and only if it vanishes except for a finite number of vertices and edges. We define $\U:C_c(\mathcal{X})\to  L^2(\T^n,\C^{n+1})$ by setting, for $f\in C_c(\mathcal{X})$ and $\xi\in\T^n$,

\begin{equation*}
(\U f)(\xi)=\(\sum_{\mu\in\Z^n}e^{-2\pi i\xi \cdot\mu}f(\mu),\sum_{\mu\in\Z^n}e^{-2\pi i\xi\cdot\mu}f(\mu\e_1),\ldots,\sum_{\mu\in\Z^n}e^{-2\pi i\xi\cdot\mu}f(\mu \e_n)\) .
\end{equation*}
 
Then $\U$ extends to a unitary operator, still denoted by $\U$, from $\ell^2(\mathcal{X})$ to $\mathcal H$. Further, we set $\mathcal H_{\mathcal V}:=\U P_{\mathcal V} (\ell^2(\mathcal X))\cong L^2(\T^n)$ and $\mathcal H_{\mathcal A}:=\U P_{\mathcal A} (\ell^2(\mathcal X))\cong L^2(\T^n,\C^{n})$. 

We draw the reader's attention to the fact that this definition of $\U$ correspond to the following choice of the Fourier transform in $\Z^n$:
 \begin{equation*}
\F:l^2(\Z^n)\to L^2(\T^n)\quad ; \quad (\F f)(\xi):=\sum_{\mu\in\Z^n}e^{-2\pi i \xi\cdot\mu}f(\mu)\ .
 \end{equation*} 
Finally, let us define  the  functions 
\begin{equation*}
a_j(\xi):=-1+e^{-2\pi i\xi_j}\ .
\end{equation*}

The following \nameCref{Pr_h_0} shows that through conjugation by $\U$, the operator $H_0$ becomes a multiplication operator, enabling the study of its spectral properties through the examination of characteristics of its band functions.

 \begin{Proposition}[{\cite[Prop. 3.5]{Pa17}}]\label{Pr_h_0}
 The operator $H_0$ satisfy that
\begin{equation*}
\U H_0\U^*=h_0
\end{equation*}
where $h_0$ denotes the multiplication operator by the real analytic function \begin{equation*}
h_0:\T^n\to M_{n+1\times n+1}(\C)
\end{equation*}
 on $L^2(\T^n,\C^{n+1})$ given by
\begin{equation}\label{h_0}
h_0(\xi)=\begin{pmatrix}
m &a_1(\xi)&\ldots &a_n(\xi)\\
\overline{a_1(\xi)}&-m &\ldots&0\\
\vdots&&\ddots&\vdots\\
\overline{a_n(\xi)}&0&\ldots&-m
\end{pmatrix}\ .
\end{equation}
\end{Proposition}

\subsection{Spectrum and resolvent of \texorpdfstring{$H_0$}{H0}}\label{subsec:Spectral}

The  band functions of  $h_0$ have an explicit expression so we are able to compute $\sigma(H_0)=\bigcup_j \lambda_j(\T^n)$. Indeed, from \cref{h_0} one can see that for $\xi\in\T^n$ the characteristic polynomial associated to $h_0(\xi)$ is given by
\begin{equation}\label{eq:characteristic}
p_{z}(\xi)=(-1)^{n}(m+z)^{n-1}\Big(m^2-z^2+\sum_{j=1 }^n|a_i(\xi)|^2\Big)\ .
\end{equation}
For convenience we define $r:\T^n\to \R^+$ and $r_i:\T^n\to \R^+$ for $i\in \{1,\dots n\}$ by
\begin{equation}\label{varrho}
r(\xi):=\sum_{j=1 }^n|a_j(\xi)|^2 \text{ and } r_i(\xi)=r(\xi)-|a_i(\xi)|^2=\sum_{j\neq i }|a_j|^2 \ .
\end{equation}
Thus, there are  three band functions:
\begin{equation}\label{eq:bandas}
z_0(\xi)=-m\ , \quad z_\pm(\xi)=\pm\sqrt{m^2+r(\xi)}.
\end{equation}
From  the identities
\begin{equation}\label{|a_j|^2}
|a_j(\xi)|^2=2(1-\cos(2\pi\xi_j))=4\sin^2(\pi\xi_j)
\end{equation}
we easily see that the spectrum of $H_0$ satisfies \cref{spectrum}. Moreover, as is shown in  \Cref{fig:bandas}, we can observe that the threshold $-m$ correspond to both the maximum of $z_-$ and the constant value of the flat band. Note from \Cref{|a_j|^2,eq:bandas} that $z_-$ attains its maximum only at $\mathbf 0\in\T^n$ for every $n$ and hence \Cref{fig:bandas} is generic. 
\begin{figure}\centering
\includegraphics[width=0.4\textwidth]{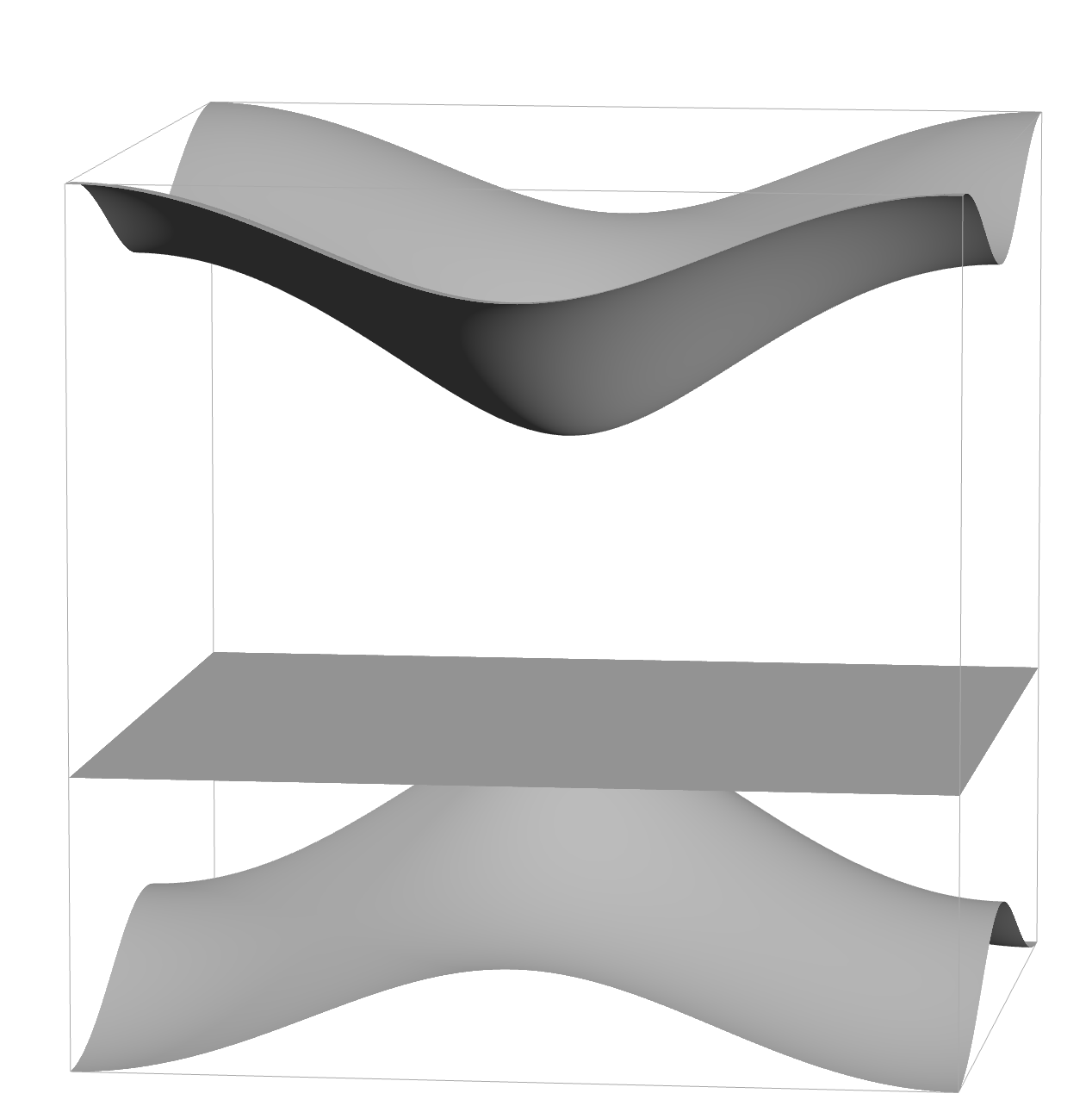}
\includegraphics[width=0.4\textwidth]{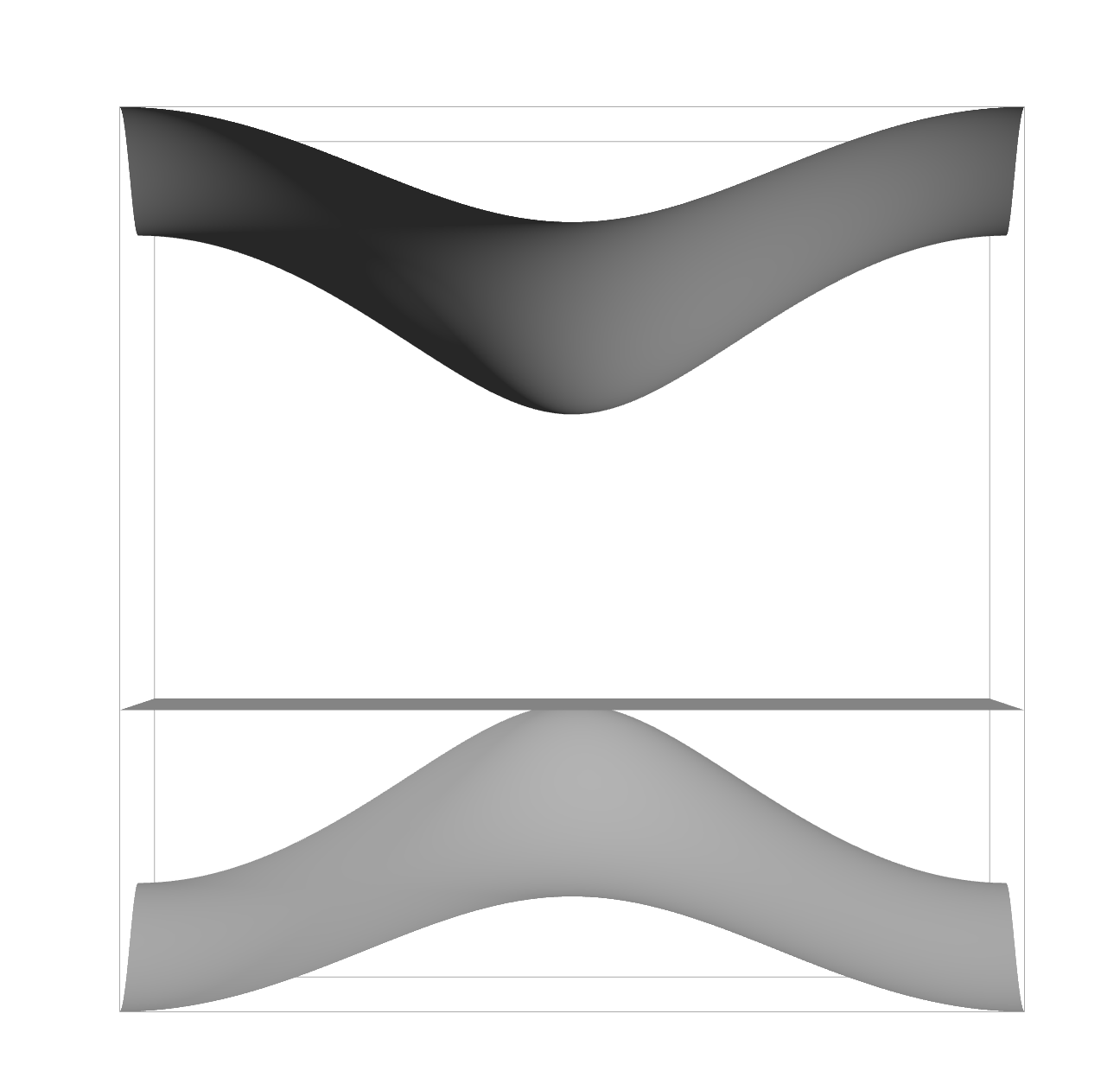}
\caption{Two views of the three band functions for $n=2$. The negative band and the flat band only touch at $(0,0)$. }\label{fig:bandas}
\end{figure}

From \Cref{h_0} and for $z \notin \sigma(H_0)$ on can check that
\begin{multline*}
(h_0-z)^{-1}=\frac{(-1)^n}{p_z}\times \\
\begin{pmatrix}
(z+m)^{n}&a_1(z+m)^{n-1} &\cdots & a_n(z+m)^{n-1}\\[.5em]
\overline{a_1}(z+m)^{n-1}&(z+m)^{n-2}(z^2-m^2-r_1)&\cdots&a_{n}\overline{a_{1}}(z+m)^{n-2}\\[.5em]
\vdots&\vdots&\ddots&\vdots\\
\overline{a_n}(z+m)^{n-1}&a_1\overline{a_n}(z+m)^{n-2}&\cdots&(z+m)^{n-2}(z^2-m^2-r_n)
\end{pmatrix}
\end{multline*}
from where one can obtain
\begin{equation}\label{resolvent}
(h_0-z)^{-1}\hspace{-2pt}=\frac{1}{(m+z)(m^2-z^2+r)}
 \begin{pmatrix}
(z+m)^{2}&a_1(z+m)&\cdots & a_n(z+m)\\[.5em]
\overline{a_1}(z+m)&z^2-m^2-r_1&\cdots&a_{n}\overline{a_{1}}\\[.5em]
\vdots&\vdots&\ddots&\vdots\\
\overline{a_n}(z+m)&a_1\overline{a_n}&\cdots&z^2-m^2-r_n
\end{pmatrix}\ .
\end{equation}

Notice that from the particular form of $h_0(\xi)+m$, that can be obtained directly from \Cref{h_0}, we can check that $\Ker (h_0+m)\leq \mathcal H_{\mathcal A}$. Indeed, one can prove directly that $\Ker (H_0+m)\leq \{\mathbf{0}\}\times \ell^2(A)$ by constructing for each $\mu\in \Z^n$ a closed path over which we define a cochain alternating the values $1$ and $-1$, see \cite[Sec. 2]{MPR23} for an explicit construction for the $\Z^2$ case. We stress that the flat bands of discrete periodic graphs are known to generate finitely supported eigenfunctions \cite{Ku91}.

\section{Perturbed operator and Main Result}\label{sec:main}


We turn now our attention to  the concrete class of perturbations $V$ that we will  treat in this article. A symmetric multiplication operator on $\ell^2(\mathcal{X})$ is defined by $V:\mathcal{X}\to \R$ such that $V(\e)=V(\overline{\e})$ for every $\e\in A$. Given such a $V$, our full hamiltonian is defined by \Cref{eq:fullH}. 

Further, we define the following real-valued functions on $\Z^n$
\begin{equation}\label{eq:littlev}
v_0(\mu):=V(\mu); \quad v_j(\mu):=V(\mu\e_j), \quad 1\leq j \leq n\ .
\end{equation}
This choice allows us to further specify the decay of $V$ at infinity, but other choices of representatives would give the same type of decay. 

Let us consider the class of Symbols  $S^\gamma(\Z^n)$ given by the functions $v:\Z^n\to \C$ that satisfies for any multi-index $\alpha=(\alpha_1,\cdots,\alpha_n)\in\N^n$
\bel{31may21} |{\rm D}^\alpha v(\mu)|\leq C_\alpha \langle \mu \rangle^{-\gamma-|\alpha|},\ee
where ${\rm D_{j}} v(\mu):=v(\mu+\delta_j)-v(\mu)$, $|\alpha|:=\sum_{j=1}^n\alpha_j$, and ${\rm D}^\alpha:={\rm D}_{1}^{\alpha_1}...{\rm D}_{n}^{\alpha_n} $.

\begin{Definition}\label{Def:admissible}
We call a perturbation $V$ \emph{admissible}  of order $\gamma$, with $n>\gamma>0$, if $\{v_j\}_{j=0}^n\in S^\gamma(\Z^n)$ and for $j=1,\dots,n$
\begin{equation}\label{eq:reduccion}
v_j(\mu)=\langle \mu \rangle ^{-\gamma} (\Gamma_j+o(1)) \text{ as }\mu\to \infty, 
\end{equation}
with  $\Gamma_j\neq0$ for at least one $j$.  
\end{Definition}

This condition may look to be restrictive, but simplifies the presentation of the results. Naturally,  alternative  classes of symbols and asymptotic behaviours at infinity of the $v_j$'s could be addressed using  akin methods to those employed in this article.

For an admissible  perturbation we define the diagonal $(n+1)\times(n+1)$ matrix $\Gamma$ by
\begin{equation}\label{eq:Gammadef}
\Gamma_{ll}=\begin{cases}
\Gamma_{l-1} &\text{ if }\Gamma_{l-1}>0\ ,\\
0&\text{ otherwise } .
\end{cases}
\end{equation}
We define as well the function  $M:\T^n\to M_{(n+1)\times(n+1)}(\mathbb C)$ by
\begin{equation}\label{eq:Mpositiva}
 M(\xi):= \frac1{r(\xi)} \begin{pmatrix}
 0&0&0&\cdots & 0\\[.5em]
  0&r_1(\xi)&-a_2(\xi)\overline{a_1(\xi)}&\cdots&-a_{n}(\xi)\overline{a_{1}(\xi)}\\[.5em]
  0&-a_1(\xi)\overline{a_2(\xi)}&r_2(\xi)&\cdots&-a_{n}(\xi)\overline{a_2(\xi)}\\
 \vdots&\vdots&\vdots&\ddots&\vdots\\
  0&-a_1(\xi)\overline{a_n(\xi)}&-a_2(\xi)\overline{a_n(\xi)}&\cdots&r_n(\xi)
\end{pmatrix}\ .
\end{equation}

\begin{Theorem}\label{T1}  Assume that $V$ is an admissible perturbation of order $\gamma$. Define the constant $\mathcal{C}$ by
\begin{equation}\label{eq:constante} 
\mathcal{C}:=\int_{\T^n}  {\rm Tr}\left( (\Gamma M(\xi) )^{\frac n\gamma} \right) d\xi\ .
\end{equation}
Let $\tau_n$ denotes the volume of the unitary sphere in $\R^n$. Then, the eigenvalue counting function satisfies
\begin{equation}\label{maintheo_1} 
\mathcal{N}(\lambda)=\lambda^{-\frac{n}\gamma}\, \,(\mathcal C\,\tau_n+o(1)),\quad \lambda\downarrow 0\ .
\end{equation}
\end{Theorem}


\begin{Remark}\label{Re:Landau}
The best--known case of degenerate eigenvalues in the continuous setting is the Landau Hamiltonian on $\R^{2}$. Although they are not usually thought of as flat bands, the direct integral decomposition obtained from the Landau gauge gives us that each Landau level is the image of a constant band function in $\R$. In this sense, it is somewhat natural that the asymptotic order  obtained in \Cref{maintheo_1} coincides with the result in \cite[Theo. 2.6]{Ra90}, see also \cref{21fev2024}. However, the constants differ in both cases. For the Landau Hamiltonian, the constant depends only on the multiplicity of the corresponding Landau level and the intensity of the magnetic field, whereas for the discrete Dirac operator, the perturbation interacts with the associated eigenspace non--trivially as encoded by $\mathcal{C}$. 
\end{Remark}

\section{Proof }\label{sec:proof}

In this section we will prove our main result \Cref{T1}. Before that, we start by recalling some known results on compact operators in order to settle notation and then reduce the study of $\mathcal N $ to the study of the eigenvalue counting function of an effective Hamiltonian.

\subsection{Some notation and results on compact operators}

Given the Hilbert spaces $\mathcal{H}_1$ and $\mathcal{H}_2$, we denote by $\mathfrak{S}_\infty(\mathcal{H}_1,\mathcal{H}_2)$ the class of compact operators from $\mathcal{H}_1$ to $\mathcal{H}_2$. When $\mathcal{H}_1=\mathcal{H}_2=\mathcal{H}$ we will just write $\mathfrak{S}_\infty(\mathcal{H})$.  For $K=K^*\in\mathfrak{S}_\infty(\mathcal{H})$ and $s>0$ we set
$$
n_{\pm}(s; K) : = \rank\mathds{1}_{(s,\infty)}(\pm K)\ .
$$
Thus, the functions $n_{\pm}(\cdot; K)$ are respectively the counting functions of the positive and negative eigenvalues of the operator $K$. For $K\in\mathfrak{S}_\infty(\mathcal{H}_1,\mathcal{H}_2)$ we define
\begin{equation*}
n_*(s; K) : = n_+(s^2; K^* K), \quad s> 0;
\end{equation*}
thus  $n_{*}(\cdot; K)$ is the counting function of the singular values of $K$ which, when ordered non--increasingly, we denote by $\{s_j(K)\}$. 
Let $K_j$, $j=1,2$ be self-adjoint compact operators. For $s_1, s_2>0$, we have  the  Weyl inequalities (see e.g. \cite[Theorem 9.2.9]{BS87})
\begin{equation}\label{weyl1}
n_\pm(s_1 + s_2; K_1 + K_2) \leq n_\pm(s_1; K_1) + n_\pm(s_2; K_2)\ .
\end{equation}
 If instead we only have $\{K_1,K_2\}\subset \mathfrak{S}_\infty(\mathcal{H}_1,\mathcal{H}_2)$ the Ky Fan inequality (see e.g. \cite[Subsection 11.1.3]{BS87}) gives
\begin{equation}\label{kyfan1}
n_*(s_1 + s_2; K_1 + K_2) \leq n_*(s_1; K_1) + n_*(s_2;K_2)\ .
\end{equation}
Further, for $0<p<\infty$  we define the class of compact operators $\mathfrak{S}_{p,w}$ by 
\begin{equation*}
\mathfrak{S}_{p,w}:=\{K\in \mathfrak{S}_\infty: s_j(K)=O(j^{-1/p})\}\ ,
\end{equation*}
together with the quasi-norm 
\begin{equation*}
||K||_{p,w}:=\sup_j\{j^{1/p}s_j(K)\}=\left(\sup_{s>0} \{s^pn_*(s;K)\right)^{1/p}
\end{equation*}
that satisfies the "weakened triangle inequality" 
\begin{equation*}
||K_1+K_2||_{p,w}\leq2^{1/p}(||K||_{p,w}+||K||_{p,w})
\end{equation*}
and the "weakened H\"older inequality" 
\begin{equation}\label{weak_holder}
||K_1K_2||_{r,w}\leq c(p,q) ||K_1||_{p,w}||K_2||_{q,w}\ ,
\end{equation}
for $r^{-1}=p^{-1}+q^{-1}$ and $c(p,q)=(p/r)^{1/p}(q/r)^{1/q}$ (see \cite[Chapter 11]{BS87}).  

Finally, consider the set $l_{p,w}$ of functions $v:\Z^n\to\C$ such that 
\begin{equation*}
\#\{\mu\in\Z^n:|v(\mu)|>\lambda\}=O(\lambda^{-p})\ .
\end{equation*}

Let us finish this section  by considering the following result, which is a particular case of  \cite[Theorem 4.8(ii)]{BKS91} 
\begin{Proposition}[Cwikel-Birman-Solomyak]\label{Cwikel-Bir-Sol} Let $p>2$ and assume  $v\in l_{p,w}$ and $f\in L^{p}(\T^n)$. Then $f\F v\in \mathfrak{S}_{p,w}( \ell^2(\Z^n),L^2(\T^n))$, and there exists a positive constant $C(n)$ such that 
$$  ||f\F v||_{p,w} \leq C(n) \| v\|_{l_{p,w}}\|f\|_{L^{p}(\T^n)}.$$
\end{Proposition}

\subsection{Effective Hamiltomian} 
In this section we will use the notation 
$$
\mathcal N((a,b);T):=\rank\mathds{1}_{(a,b)}(T),
$$
where $a<b$ and  $T$ is a self-adjoint operator without essential spectrum in $(a,b)$. Following the approach coming from the study of magnetic Schrödinger operators, our aim is to study $PVP$ where $P$ stands for the projection on the flat band, i.e.,
\bel{def_P}
P:=\mathds{1}_{\{-m\}}(H_0)\ .
\ee

\begin{Lemma}\label{Le4} Recall that $M:\T^n\to M_{(n+1)\times(n+1)}$ was defined in \Cref{eq:Mpositiva}.
Then
$$
P=\U^* M \U\ .
$$
\end{Lemma}
\begin{proof}
By Stone formula one can check that 
\begin{equation*}
P=\slim_{\kappa\downarrow 0}\frac{1}{i\pi}\int_{-m}^0\bigl((H_0-s-i\kappa)^{-1}-(H_0-s+i\kappa)^{-1}\bigr)ds\ .
\end{equation*}
Then, from \Cref{resolvent} we need only to check that
\begin{equation*}
\lim_{\kappa\downarrow 0}\int_{-m}^0\biggl(\frac{1}{m+\lambda+i\kappa}-\frac{1}{m+\lambda-i\kappa}\biggr)d\lambda=-i\pi  \delta_{-m},
\end{equation*}
where $\delta_{-m}$ is the Dirac delta function on $-m$.
\end{proof}
Now, set $P^\perp:=I-P$ and for $\kappa>0$ define 
$$
H_\kappa^\pm:=H_0+P(V\pm\kappa|V|)P+P^\perp(V\pm\kappa^{-1}|V|)P^\perp
$$
Then, by \cite[Lemma 4.2]{PuRo11}
\bel{22fev2024}
H_\kappa^-\leq H \leq H_\kappa^+.
\ee

Then, arguing as in the proof of \cite[Theorem 4.1(ii)]{PuRo11} we obtain that:



\begin{align}
\pm\mathcal N(\lambda)\leq&\pm\mathcal N((-m+\lambda,0); -mP+P(V\pm\kappa|V|)P)\nonumber \\
\label{20fev2024b}&\pm \mathcal N((-m+\lambda,0); P^\perp(H_0+(V\pm\kappa^{-1}|V|))P^\perp)+O(1).
\end{align}

In the next \nameCref{Le1} we treat the second term on the right of \Cref{20fev2024b}, showing that the perturbation $V$ interacts with the complement of the degenerated eigenspace only at a lesser order. 

\begin{Lemma}\label{Le1} 
$$\mathcal N((-m+\lambda,0); P^\perp(H_0+V\pm\kappa^{-1}|V|)P^\perp)=o(\lambda^{-n/\gamma}),\quad \lambda\downarrow 0.$$
\end{Lemma}

\begin{proof} Define the function $W:\mathcal X\to \R^+$ by 
$w_j(\mu)=\langle \mu\rangle^{-\gamma}$, where we are using the notation of \cref{eq:littlev}.  From \cref{31may21}  there exist a constant $C>0$ such that $|V|\leq C W$.  Denote by $W_\kappa:=C(1+\kappa^{-1}) W$. Then it can been seen that (again as in the proof of [Theorem 4.1(ii)]\cite{PuRo11})
$$
\mathcal N((-m+\lambda,0); P^\perp(H_0+V\pm\kappa^{-1}|V|)P^\perp)\leq \mathcal N((-m+\lambda,0); P^\perp(H_0+W_\kappa)P^\perp)+O(1).
$$
Now, by the Birman-Schwinger principle (see for instance \cite{Kla82,Push09}), we get for $\lambda\in(0,m)$ \small
\begin{equation*}
\mathcal N((-m+\lambda,0); P^\perp(H_0+W_\kappa)P^\perp)=n_+(1;P^\perp W_\kappa^{1/2}P^\perp (H_0+m-\lambda)^{-1}P^\perp W_\kappa^{1/2}P^\perp)+O(1).
\end{equation*}
\normalsize
Define the $(n+1)\times (n+1)$ matrix 
$$
 M_{R}:= \begin{pmatrix}
\lambda &a_1&\cdots&a_n \\[.5em]
\overline{a_1}&\lambda-2m&\cdots&0\\[.5em]
\vdots&\vdots&\ddots&\vdots\\
\overline{a_n}&0&\cdots&\lambda-2m
\end{pmatrix}.
$$

Then, from \cref{resolvent} and  \cref{Le4}  it is not difficult to  see that 
for $\lambda\in(0,m)$
\begin{equation*}
 (H_0+m-\lambda)^{-1}P^\perp= 
\U^*\frac{ {M}_{R}}{r+\lambda(2m-\lambda)} ({\rm Id}-M)\U ,
\end{equation*}
where ${\rm Id}$ denotes the identity $(n+1)\times(n+1)$  matrix. Furthermore, the operator $W_\kappa^{1/2}P^\perp (H_0+m-\lambda)^{-1}P^\perp W_\kappa^{1/2}$ is obviously compact and from 
\cref{weak_holder}
\begin{multline*}\label{15jun21bis}
\|W_\kappa^{1/2}P^\perp (H_0+m-\lambda)^{-1}P^\perp W_\kappa^{1/2}\|_{n/\gamma,w}\leq C \|W_\kappa^{1/2}\U^*\frac{1}{r+\lambda(2m-\lambda)}\|_{2n/\gamma,w} \\
\times \|{M}_R({\rm Id}- M)\U W_\kappa^{1/2}\|_{2n/\gamma,w}.
\end{multline*}
\normalsize
Consider the operator $ W_\kappa^{1/2}\U^*\frac{1}{r+\lambda(2m-\lambda)}$. Since $\frac{1}{r+\lambda(2m-\lambda)}$ is bounded, it is in $L^p(\T^n)$ for any $p>1$. Further,   each component of  the multiplication operator  $W_\kappa^{1/2}$ is in $l_{2n/\gamma,w}$. Then, since $2n/\gamma>2$,  by \cref{Cwikel-Bir-Sol},
$$\left\| W_\kappa^{1/2}\U^*\frac{1}{r+\lambda(2m-\lambda)} \right\|_{2n/\gamma,w}\leq C \left\| \frac{1}{r+\lambda(2m-\lambda)} \right\|_{L^{\frac{2n}\gamma}(\T^n)}\|W_\kappa^{1/2}\|_{\oplus_{l=0}^nl_{2n/\gamma},w}.$$
To estimate  the $L^{\frac{2n}\gamma}$ norm we use the coarea formula 
\begin{align*}\int_{\T^n} |(r+\lambda(2m-\lambda))^{-1}|^{2n/\gamma}=&\int_{0}^{1/2}\frac{1}{(\rho+\lambda(2m-\lambda))^{2n/\gamma}}\int_{r(\xi)=\rho}\frac{ 1}{|\nabla r(\xi)|} d\xi d\rho 
\\ \leq &C   \int_{0}^{1/2}\frac{1}{(\rho+\lambda(2m-\lambda))^{2n/\gamma}} \int_{r(\xi)=\rho}\frac{ 1}{r(\xi)^{1/2}} d\xi d\rho 
\\ \leq &C   \int_{0}^{1/2}\frac{1}{(\rho+\lambda(2m-\lambda))^{2n/\gamma}} \frac{1}{\rho^{1/2}}\int_{r(\xi)=\rho} d\xi d\rho 
\\ \leq &C   \int_{0}^{1/2}\frac{\rho^{n/2-1}} {(\rho+\lambda(2m-\lambda))^{2n/\gamma}}  d\rho 
\\ \leq &  C \lambda^{-2n/\gamma+n/2},
\end{align*}
where in the first and third inequalities  have used \cref{|a_j|^2}. 
Analogously,  $$\left\| {M}_R({\rm Id}-M)\U W_\kappa^{1/2}  \right\|_{2n/\gamma,w}\leq C \|W_\kappa^{1/2}\|_{\oplus_{l=0}^nl_{2n/\gamma},w},$$ 
since the matrix ${M}_R({\rm Id}- M)$  is bounded with uniform bound in $\lambda$. 
Putting all this together we obtain 
$$\|W_\kappa^{1/2}P^\perp (H_0+m-\lambda)^{-1}P^\perp W_\kappa^{1/2}\|_{n/\gamma,w}\leq  C \lambda^{\gamma/4-1},$$
which is equivalent to say  that 
\begin{align*}n_*(s; W_\kappa^{1/2}P^\perp (H_0+m-\lambda)^{-1}P^\perp W_\kappa^{1/2})\leq 
 & C \lambda^{-n/\gamma+n/4}\\
 =&o(\lambda^{-n/\gamma}).\qedhere
\end{align*}
\end{proof}

\subsection{Eigenvalue counting function for the effective Hamiltonian}

From \cref{20fev2024b,Le1} 
\begin{equation}\label{21fev2024}
\pm\mathcal N(\lambda)\leq \pm \mathcal N((\lambda,m); P(V\pm\kappa|V|)P)+o(\lambda^{-n/\gamma}), \quad \lambda\downarrow 0\ .
\end{equation}
Then, we are led to study the distribution of positives eigenvalues of the compact operator $P(V\pm\kappa|V|)P$. 

For ease of notation, for any $\kappa >0$  we define $T_\kappa^\pm$ in $\mathfrak{S}_\infty(\mathcal{H})$ by
 \begin{equation*}T^\pm_\kappa:=\U P(V\pm\kappa|V|)P\U^*=M\U(V\pm\kappa|V|)\U^*M\ .
\end{equation*}

\begin{Proposition}\label{T2} 
For an admisible $V$ we have 
 $$ n_+(\lambda;T^\pm_\kappa)=\Big(\frac{1\pm\kappa}{\lambda}\Big)^{n/\gamma}\tau_n\int_{\T^n}  {\rm Tr}\left( {(M(\xi)\Gamma{M(\xi)})^{n/\gamma}} \right) d\xi\,(1+o(1)),\quad \lambda\downarrow 0\ .$$ 
\end{Proposition}

In order to proof this \nameCref{T2} we follow the ideas of \cite[Theorem 6.1]{MPR23}, which  in turn are inspired by the proof of  \cite[Theorem 1]{BS70}. By analogy, we denote $\Box:=[0,1)^n\subset\R^n$ and hence $\mathcal H\cong \oplus^n_{j=0}L^2(\Box)$. Finally, for ease of notation, let us set $\Ve=\U(V\pm\kappa|V|)\U^*$.

\begin{Remark}
The statement of \Cref{T2} is particular to our effective Hamiltonian and problem. However, in the proof we use only that $M\in L^p(\T^d;\C^{n+1})$ for $p>2$ and we could also replace $V\pm\kappa|V|$ with another potential satisfying \Cref{31may21,eq:reduccion}. A similar statement holds for $n_-$.
\end{Remark}

\begin{Lemma}\label{2} Let $X$ and $Y$ be two subsets of $\Box$ with no interior points in common. Then $$n_*(r; \mathds{1}_X\Ve\mathds{1}_Y)=o(r^{-n/\gamma}),\quad r\downarrow 0.$$
\end{Lemma}
\begin{proof} The proof uses \cref{Cwikel-Bir-Sol} and is almost equal to the proof of  \cite[Lemma 6.4]{MPR23}.
\end{proof}

\begin{Lemma}\label{3} Let $\{\Box_j\}$ be a partition of $\Box$ into cubes of equal size $1/ q^n$, $q\in\Z_+$,  and let $\{B_j\}_{j=1}^{q^n}$ be  matrices in $M_{(n+1)\times(n+1)}(\C)$.
Let $\check T_\kappa^\pm:\oplus^n_{j=0}L^2(\Box)\to \oplus^n_{j=0}L^2(\Box)$ be the operator defined by 
$$
\check T^\pm_\kappa=\sum_j B_j \mathds{1}_{
\Box_j}  \Ve\mathds{1}_{\Box_j} B_j^*.
$$
Then, for any $\delta \in (0,1)$, 
\begin{align*}
&\frac{\tau_n}{q^n}\sum_{j}
{\rm Tr}((B_j\Gamma(1\pm\kappa) B_j^*)^{n/\gamma})
\,(\lambda(1+\delta))^{-n/\gamma} (1+o(1))\\
\leq &n_+(\lambda;\check{T}^\pm_\kappa)\\
  \leq &
\frac{\tau_n}{q^n}\sum_{j}
{\rm Tr}((B_j\Gamma(1\pm\kappa) B_j^*)^{n/\gamma})
\,(\lambda(1-\delta))^{-n/\gamma} (1+o(1)),\quad \lambda\downarrow 0.
\end{align*}
\end{Lemma}
\begin{proof} We will show the proof of the upper bound. The lower bound is similar. Let $B_0$ be a constant $(n+1)\times(n+1)$ matrix. Then for any $\delta\in(0,1)$
\begin{align*}
n_+(\lambda;B_0\Ve B_0^*)\geq  &\sum_jn_+(\lambda(1+\delta);B_0\mathds{1}_{
\Box_j} \Ve \mathds{1}_{
\Box_j}  B_0^*)
-n_-(\lambda\delta;\sum_{j\neq l}B_0\mathds{1}_{
\Box_j} \Ve\mathds{1}_{
\Box_l}  B_0^*)\\
=&q^n n_+(\lambda(1+\delta);B_0\mathds{1}_{
\Box_0} \Ve\mathds{1}_{
\Box_0}  B_0^*) +o(\lambda^{-n/\gamma}),
\end{align*}
where for the inequality we used \Cref{weyl1}. For the equality we used first  the fact that each operator $B_0\mathds{1}_{
\Box_j} \Ve\mathds{1}_{
\Box_j}  B_0^*$ is unitary equivalent to $B_0\mathds{1}_{
\Box_0} \Ve\mathds{1}_{
\Box_0}  B_0^*$, for $\Box_0=(0,1/q)^n$. Then we used and \Cref{2,kyfan1}. It follows that 

\begin{equation}\label{eq:primeracota}
n_+(\lambda;B_0\mathds{1}_{
\Box_0} \Ve\mathds{1}_{
\Box_0}  B_0^*) \leq \frac{1}{q^n} n_+(\lambda(1-\delta);B_0\Ve B_0^*)+o(\lambda^{-n/\gamma})\ .
\end{equation}

Define $v_*(\mu)=\langle \mu\rangle^{-\gamma}$.  One can check that 
\begin{equation}\label{18nov22}
\#\{\mu\in\Z^n:v_*(\mu)>\lambda\}=\tau_m \lambda^{-n/\gamma}(1+o(1)),\quad \lambda\downarrow 0\ ,
\end{equation}
see for instance \cite[Prop. 2 XIII.15]{RS78}.

From \cref{eq:reduccion} set $\hat V^\pm_0:=\U \Gamma(1\pm\kappa) \langle \mu \rangle^{-\gamma}\U^*$ and use the Weyl inequalities \cref{weyl1}  to obtain that for $\tilde \delta\in(0,1)$
$$
n_+(\lambda;B_0 \Ve B_0^*)\\
\leq n_+
(\lambda(1-\tilde{\delta});B_0 \hat V^\pm_0 B_0^*)+ n_+(\lambda\tilde{\delta};B_0 (\Ve-\hat V^\pm_0) B_0^*).
$$
Now, denote  by $\{\beta_{0,l}\}$ the eigenvalues of the matrix $B_0\Gamma B_0^*$. We have that  the eigenvalues of $B_0 \hat V_0^\pm B_0^*$ are given by 
$$
\{\beta_{0,l}\, v_*(\mu):1\leq l\leq k, \mu \in \Z^d \}.
$$
Thus \cref{18nov22} implies that
\begin{align*}
n_+(\lambda;B_0 \hat V_0^\pm B_0^*)&=\#\{ 1\leq l\leq k, \mu \in \Z^n:  \beta_{0,l}\, v_*(\mu)>\lambda \}\\
&=\sum_{\beta_{0,l}>0} n_+(\lambda/\beta_{0,l};v_*)\\
&=\Big(\tau_n\sum_{\beta_{0,l}>0} \beta_{0,l}^{n/\gamma}\Big) \lambda^{-n/\gamma} (1+o(1)),\quad \lambda\downarrow 0.
\end{align*}

The same reasoning can be used to show that $n_+(B_0(\Ve-\hat V_0)B^*_0)=o(\lambda^{-n/\gamma}) $. 
Putting the previous inequalities together, for all $\delta, \tilde \delta\in(0,1)$ \begin{align*}
n_+(\lambda;\check T^\pm_\kappa)&=\sum_{j}n_+(\lambda;B_j \mathds{1}_{
\Box_j}  \Ve \mathds{1}_{\Box_j} B_j^*)\\
&\leq \frac1{q^n}\sum_{j}n_+(\lambda(1-\delta);B_j \Ve  B_j^*)+o(\lambda^{-n/\gamma})
\\
&\leq \frac1{q^n}\sum_{j}n_+\Big(\frac{\lambda(1-\delta)}{(1+\tilde{\delta})};B_j \hat V_0^\pm B_j^*\Big)+o(\lambda^{-n/\gamma})\\
&= \frac{\tau_n}{q^n}\sum_{j}
{\rm Tr}\Big((B_j\Gamma(1\pm\kappa) B_j^*)^{n/\gamma}\Big)
\,\left(\frac{\lambda(1-\delta)}{(1+\tilde{\delta})}\right)^{-n/\gamma} (1+o(1)),\quad \lambda\downarrow 0\ .\qedhere
\end{align*}

\end{proof}

\begin{proof}[Proof of \Cref{T2}]

Let $\varepsilon>0$, and take $B_\varepsilon=\sum_j B_{\varepsilon,j}\mathds{1}_{
\Box_{\varepsilon,j} } $   a  step matrix function such that 
$ \|M-B_{\varepsilon}\|_{L^p(\T^n)}<\varepsilon$. Assume that the size of each cube $\Box_{\varepsilon,j} $ is $1/q^n$ as in the previous lemma.

Take 
$
S_\varepsilon:=B_\varepsilon \Ve B_\varepsilon^*
$. Then by \cref{Cwikel-Bir-Sol} $\|T_\kappa^\pm-S_\varepsilon\|_{n/\gamma,w}<C\varepsilon$, which means that 
\bel{11jan24}
 n_*(\lambda; T_\kappa^\pm-S_\varepsilon  ) \leq (C\varepsilon)^{n/\gamma} \lambda^{-n/\gamma}.
\ee
Also, let $\check{T}_{\varepsilon,\kappa}^\pm=\sum_{j}B_{\varepsilon,j}\mathds{1}_{
\Box_{\varepsilon,j} } \Ve B_{\varepsilon,j}\mathds{1}_{
\Box_{\varepsilon,j} }$. Thus,  by \cref{2} 
\bel{11jan24b}
n_*(s;S_{\varepsilon}-\check{T}_{\varepsilon,\kappa}^\pm)=o(s^{-n/\gamma}).
\ee
Now, using \cref{3}, we have that for 
any $\delta \in (0,1)$
\begin{align}\label{11jan24c}
&\frac{\tau_n}{q^n}\sum_{j}
{\rm Tr}((B_{\varepsilon,j}\Gamma(1\pm\kappa) B_{\varepsilon,j}^*)^{n/\gamma})
\,(\lambda(1+\delta))^{-n/\gamma} (1+o(1))\nonumber\\
\leq &n_+(\lambda;\check{T}_{\varepsilon,\kappa}^\pm)\\
 \leq &
\frac{\tau_n}{q^n}\sum_{j}
{\rm Tr}((B_{\varepsilon,j}\Gamma(1\pm\kappa) B_{\varepsilon,j}^*)^{n/\gamma})
\,(\lambda(1-\delta))^{-n/\gamma} (1+o(1)),\quad \lambda\downarrow 0\nonumber\ .
\end{align}
Finally, putting together \cref{11jan24,11jan24b,11jan24c,weyl1,kyfan1}, and making $\lambda$, $\delta$ and  $\varepsilon$ goes to $0$, we finish the proof.
\end{proof}

\begin{proof}[Proof of \Cref{T1}]
The result follows from \Cref{T2} by taking $\kappa\downarrow 0$, \Cref{21fev2024}  and using the cyclicity of the trace.
\end{proof}

\section{The Laplacian on a particular \texorpdfstring{$\Z^2$}{Z2}-periodic graph}\label{sec:Laplace}

\subsection{A simple example of a \texorpdfstring{$\Z^2$}{Z2}-periodic graph with a flat band}

Let us start by briefly recalling some notions from the periodic graph theory, we refer to \cite{Su13,KS14,PR18} for more details. We say that a graph is $\Z^d$-- periodic if it admits an action of $\Z^d$ by graph--automorphisms. By fixing representatives of each orbit of vertices for this action we can define the entire part of a vertex by $\lfloor x\rfloor\check{x}=x$ where $\check{x}$ is the representative of the orbit of $x$. Then, the index of an oriented edge $\e=(x,y)$ is just $\eta(\e)=\lfloor y \rfloor - \lfloor x \rfloor $.  Note that $\eta$ is $\Z^d$--periodic and hence we can refer to the index of an edge in the quotient graph. 

Let us now denote by $\tilde{\mathcal X}=(\tilde{\V},\tilde{\A})$ the graph obtained from $\Z^2$ by adding a vertex on each edge with trivial weights (see \Cref{fig:grafoperiodico}).
The quotient graph obtained by the action of $\Z^2$ is composed by three vertices and four edges as presented in \Cref{fig:grafochico}. If we takes a representatives the vertices $(0,0)$, $(0,\tfrac12)$ and $(\tfrac12,0)$ One can easily check that $\eta (\e_1)=\eta(\e_2)=(0,0)$ while $\eta (\e_3)=(1,0)$ and $\eta(\e_4)=(0,1)$. 

\begin{figure}[h]
\begin{subfigure}{0.4\textwidth}
\begin{tikzpicture}
\draw (0,0) grid (5,5);
\draw[line width=3pt, line cap=round, dash pattern=on 0pt off 1cm](0,0) grid (5,5);
\draw[line width=3pt, line cap=round, dash pattern=on 0pt off 1cm](0.5,0) grid (4.5,5);
\draw[line width=3pt, line cap=round, dash pattern=on 0pt off 1cm](0,0.5) grid (5,4.5);
\end{tikzpicture}
\caption{The periodic graph obtained from $\Z^2$ by adding a vertex to each edge.}
\label{fig:grafoperiodico}
\end{subfigure}
\hfill
\begin{subfigure}{0.5\textwidth}
\vspace*{10pt}
\hspace*{60pt}
\begin{tikzpicture}
\node[shape=circle,draw=black] (A) at (1,1) {$x_{0,0}$};
\node[shape=circle,draw=black] (B) at (4,1) {$x_{1,0}$};
\node[shape=circle,draw=black] (C) at (1,4) {$x_{0,1}$};
\path [-] (A) edge[bend right=30, below] node {$\e_1$} (B);
\path [-] (A) edge[bend right=30, right] node {$\e_4$} (C); 
\path [-] (A) edge[bend left=30, above] node {$\e_3$} (B);
\path [-] (A) edge[bend left=30, left] node {$\e_2$} (C); 
\end{tikzpicture}
\caption{The quotient graph by the usual action of $\Z^2$.}
\label{fig:grafochico}
\end{subfigure}
\end{figure}
 Set $\tilde{H}_0=-\Delta_0$, where $\Delta_0$ is the usual graph Laplacian, \ie, for $f \in \ell^2(\tilde{\V})$ and $x\in \tilde{\V}$:
\begin{equation*}
(\tilde{H}_0f )(x)=\sum_{\e\in \tilde{\A}, \e=(x,y)} f(x)-f(y)\ .
\end{equation*}
Hence, by defining $\tilde{a}_j=1+e^{2\pi i \xi_j}$, for $j=1,2$, we obtain the following representation of the graph Laplacian as a matrix-valued multiplication operator. 
\begin{Proposition}[{\cite[Prop. 4.7]{PR18}}]\label{prop:laplace}
There exists a unitary operator $\tilde{\U}:\ell(\tilde{\V})\to L^2(\T^2;\C^n)$ such that 

\begin{equation*}
\tilde{\U} (\tilde{H}_0) \tilde{\U}^*=\tilde{h}_0
\end{equation*}
where $h_0$ denotes the multiplication operator by the real analytic function

\begin{equation*}
\tilde{h}_0:\T^2\to M_{3\times 3}(\C)
\end{equation*}

on $L^2(\T^2,\C^{3})$ given by

\begin{equation}\label{h_0laplace}
\tilde{h}_0(\xi)=\begin{pmatrix}
4&-\tilde{a}_1(\xi)&-\tilde{a}_2(\xi)\\
-\overline{\tilde{a}_1(\xi)}&2&0\\
-\overline{\tilde{a}_2(\xi)}&0&2
\end{pmatrix}\ .
\end{equation}
\end{Proposition}

Setting as before $\tilde{r}(\xi)=|\tilde{a}_1(\xi)|^2+|\tilde{a}_2(\xi)|^2$, and noticing
\begin{equation*}
|\tilde{a}_j(\xi)|^2=2+2\cos(2\pi\xi_j)=4\cos^2(\pi\xi_j)
\end{equation*}
we can obtain the associated characteristic polynomial to $\tilde{h}_0$
\begin{equation*}
\tilde{p}_z(\xi)=(2-z)(z^2-6z+8-\tilde{r}(\xi))
\end{equation*}
and the corresponding non constant band functions
\begin{equation*}
\tilde{z}_\pm(\xi)=3\pm\sqrt{1+\tilde{r}(\xi)} \ .
\end{equation*}

It follows that the spectrum satisfies 
\begin{equation}\label{spectrumlaplace}
\sigma(\tilde{H}_0)=\sigma_{ess}(\tilde{H}_0)=\sigma_{ac}(\tilde{H}_0) 
=[0,2]\bigcup[4,6]
\end{equation} 
with $2$ an embedded degenerated eigenvalue. Given $\tilde{V}:\tilde{\V}\to \R$ we define the Schrödinger operator 
\begin{equation*}
\tilde{H}=\tilde{H}_0+\tilde{V}
\end{equation*}
and the corresponding eigenvalue counting function by 

\begin{equation*}
\tilde{\mathcal{N}}(\lambda)={\rm Rank} \mathds{1}_{(2+\lambda,3)}(H)\ ,
\end{equation*}
for $\lambda\in (0,1)$. As before, by taking the limit $\lambda\downarrow 0$ we will be able to study the accumulation of eigenvalues near the perturbed flat band.  

\begin{Remark} 
An attentive reader can wonder why this Laplacian operator show the same spectral properties than the Dirac operator studied in previous sections. From a purely computational point of view, the similarities with $H_0$ can be deduced from the fact that the symbol on $\T^2$ of $\tilde{H}_0-3$ correspond to the symbol of $H_0$ with $m=1$ by replacing $a_j$ with $-\tilde{a}_j$. In general, one can say that the clear distinction of the order of a differential operator gets muddy in the discrete case, see for instance the discussion related to the continuum limit of discrete Dirac operators \cite{Na24,CGJ22}.
\end{Remark}

\subsection{Admissible perturbations and eigenvalue asymptotics.}
Let us start by noticing that for every $\mu\in\Z^2$ we can define $f_\mu\in\ell^2(\tilde{\V})$ by
\begin{equation*}
f_\mu(x)=\begin{cases}
1 & \text{ if } x=\mu+(\tfrac12,0)\ ,\\
-1 & \text{ if } x=\mu+(0,\tfrac12)\ ,\\
0 & \text{ else.}
\end{cases}
\end{equation*}
and it satisfies $H_0f_\mu=2f_\mu$. Hence, if we decompose $\ell^2(\tilde{\V})$ by 
\begin{equation*}
\ell^2(\tilde{\V})\cong \ell^2(\Z^2)\oplus \ell^2(\Z^2+(\tfrac12,0))\oplus \ell^2(\Z^2+(0,\tfrac12))
\end{equation*}
we have that $$\Ker (H_0-2)\leq\{ \mathbf{0}\}  \oplus \ell^2(\Z^2+(\tfrac12,0))\oplus \ell^2(\Z^2+(0,\tfrac12))\ .$$

Then, if we define $\tilde{v}_j:\Z^2\to\R$, for $j\in\{0,1,2\}$ by 
\begin{equation*}
\tilde{v}_0(\mu)=\tilde{V}(\mu)\quad , \quad
\tilde{v}_1(\mu)=\tilde{V}(\mu+(\tfrac12,0))\quad \text{ and } \quad
\tilde{v}_2(\mu)=\tilde{V}(\mu+(0,\tfrac12))\ ,
\end{equation*}
we can apply \Cref{Def:admissible} to $\tilde{V}$. 

Let us now observe that for any $z\notin\sigma(\tilde{H}_0)$

\begin{multline*}
(\tilde{h}_0-z)^{-1}=
\frac{1}{p_z}\begin{pmatrix}
(2-z)^2& \tilde{a}_1(2-z)&\tilde{a}_2(2-z) \\
\overline{\tilde{a}_1}(2-z)&(2-z)(4-z)-|a_2|^2&\overline{\tilde{a}_1}\tilde{a}_2\\
\overline{\tilde{a}_2}(2-z)&\tilde{a}_1\overline{\tilde{a}_2}&(2-z)(4-z)-|a_1|^2
\end{pmatrix}\\
=
\frac{1}{(z^2-6z+8-\tilde{r})}\begin{pmatrix}
(2-z)& \tilde{a}_1&\tilde{a}_2 \\
\overline{\tilde{a}_1}&(4-z)&0\\
\overline{\tilde{a}_2}&0&(4-z)
\end{pmatrix}+
\frac{1}{p_z}\begin{pmatrix}
0& 0&0 \\
0&-|a_2|^2&\overline{\tilde{a}_1}\tilde{a}_2\\
0&\tilde{a}_1\overline{\tilde{a}_2}&-|a_1|^2
\end{pmatrix}\ .
\end{multline*}

Hence, we define $\tilde{M}:\T^2\to M_{3\times 3}(\C)$ by
\begin{equation}\label{eq:Mlaplace}
 \tilde M:=\frac{1}{\tilde r}\begin{pmatrix}
 0&0&0\\
 0&|\tilde{a}_2|^2&-\tilde{a}_2\overline{\tilde{a}_1}\\
 0&-\tilde{a}_1\overline{\tilde{a}_2}&|\tilde{a}_1|^2
\end{pmatrix}\ .
\end{equation}

\begin{Theorem}\label{theo:nonmaintheo}  Assume that $\tilde V$ is an admissible perturbation of order $\gamma$ and associate $3 \times 3$ matrix $\tilde \Gamma$. Define the constant $\tilde{\mathcal{C}}$ by
\begin{equation}\label{eq:constantenonmain} 
\tilde{\mathcal{C}}:=\int_{\T^2} {\rm Tr}\left( (\Gamma\tilde M(\xi) )^{\frac n\gamma} \right) d\xi\ .
\end{equation}
Then, the eigenvalue counting function satisfies
\begin{equation}\label{eq:nonmain} 
\tilde{\mathcal{N}}(\lambda)=\lambda^{-\frac{n}\gamma}\, \,(\tilde{\mathcal{C}}\,\tau_n+o(1)),\quad \lambda\downarrow 0\ .
\end{equation}
\end{Theorem}

\subsection*{Acknowledgments}
P. Miranda was supported by the Chilean Fondecyt Grant 1201857. D. Parra was partially supported by the Chilean Fondecyt Grant 3210686 and Universidad de La Frontera, Apoyo PF24-0027. Both authors gratefully acknowledge the hospitality of the \emph{Institut de Mathématiques de Bordeaux} were the final draft of this manuscript was prepared. 

\printbibliography

\end{document}